\documentclass[a4paper,12pt]{amsart} 

\usepackage{amssymb,amsmath}
\usepackage{pb-diagram}
\usepackage{pb-xy}
\usepackage[cmtip,arrow]{xy}

\newtheorem{thm}{Theorem}[section] 
\newtheorem*{thm*}{Theorem}
\newtheorem{prop}[thm]{Proposition}
\newtheorem{cor}[thm]{Corollary}

\theoremstyle{definition} 
 
\newtheorem{rem}[thm]{Remark} 
\newtheorem{exa}[thm]{Example}

\numberwithin{equation}{section}

\newcommand{\alg}[1]{\mathfrak{#1}}
\newcommand{\algg}{\mathfrak{g}}

\begin{document}

\title{Examples of minimal $G$--structures induced by the Lee form}
\author{Kamil Niedzia\l omski}
\date{\today}

 
\address{
Department of Mathematics and Computer Science \endgraf
University of \L\'{o}d\'{z} \endgraf
ul. Banacha 22, 90-238 \L\'{o}d\'{z} \endgraf
Poland
}
\email{kamiln@math.uni.lodz.pl}

\begin{abstract}
We compute the condition of minimality of a G-structure for the Gray--Hervella class $\mathcal{W}_4$ of almost hermitian manifolds and  $\mathcal{C}_5$ class of almost contact metric structures. We also consider $\mathcal{C}_4$ class by comparison with the Grey--Hervella class $\mathcal{W}_4$. The common feature is the existence of the Lee form $\theta$ representing these structures. We show that these classes contain minimal G-structures. Here, minimality means minimality of a $G$--structure inside oriented orthonormal frame bundle $SO(M)$ of a Riemannian manifold $M$.
\end{abstract}

\maketitle

\section{Introduction}

Existence of a geometric structure compatible with a Riemannian metric on a manifold is equivalent to reduction of the structure group of (oriented) orthonormal frame bundle ($SO(M)$) $O(M)$ to certain subgroup $G\subset O(n)$, $n=\dim M$. For example, almost hermitian structure is defined by the unitary group $U(n)\subset SO(2n)$, almost contact structure by $U(n)\times 1\subset SO(2n+1)$ or  an almost quaternion--hermitian structure by $Sp(n)Sp(1)\subset SO(4n)$. Considering additionally the Levi--Civita connection $\nabla$ we may ask if this connection is compatible with the given reduction. The failure is measured by the intrinsic torsion. In particular, if the intrinsic torsion vanishes, the holonomy algebra is contained in the Lie algebra $\algg$ of the structure group.

We may classify intrinsic torsion with respect to the action of $G$ obtaining irreducible components, often called Grey--Hervella classes. Another approach was initiated by Wood \cite{wo1,wo2} and, in general case, by Martin--Cabrera and Gonzalez--Davilla \cite{gmc1} by studying harmonicity of induced section of certain homogeneous associated bundle. Is such case, we call $G$--structure harmonic. In \cite{kn} the author studied properties of intrinsic torsion by considering extrinsic geometry of a reduction $P$ inside $SO(M)$. The Riemannian metric on $SO(M)$ is induced from Riemannian metric on $M$ and Killing form on $SO(n)$. If $P$ is a minimal submanifold in $SO(M)$ we call $G$--structure minimal. 

In this note we provide examples of minimal $G$--structures for $G=U(n)$ and $G=U(n)\times 1$. More precisely, we consider $G$--structures for above mentioned $G$'s, which are locally conformally integrable, i.e. locally there is a conformal deformation of a Riemannian metric for which the induced Levi--Civita connections is a $G$-connection. This implies existence of global vector field and by duality, global $1$--form $\theta$ called the Lee form. Conditions for minimality become differential equations on $\theta$. We find examples of $G$--structures, which satisfy condition of minimality. These include Hopf manifolds for $G=U(n)$ (then the Lee form is parallel) and Kenmotsu manifolds for $G=U(n)\times 1$.

We begin by recalling basic information about intrinsic torsion, harmonicity and minimality of $G$--structures. Then we compute the minimality condition for locally conformally K\"ahler class $\mathcal{W}_4$ and class $\mathcal{C}_5$ of contact metric structures. We conclude providing appropriate examples. In a sense, this note provides a completion of the article \cite{kn} by the author, in which minimality of $G$--structures was considered but in which there were examples of such structures only in the case of almost product structures, i.e. for $G=SO(m)\times SO(n-m)$.

\section{Minimal $G$--structures via the intrinsic torsion}

All the information in this section can be found in \cite{gmc1} and \cite{kn}. Let $(M,g)$ be an oriented Riemannian manifold. Consider an oriented orthonormal frame bundle $SO(M)$. Let $\nabla$ denote the Levi--Civita connection of $g$. It induces the horizontal distribution $\mathcal{H}\subset TSO(M)$. Any vector $X\in TM$ has a unique lift $X^h_p$ to $\mathcal{H}_p$, $p\in SO(M)$. Vertical distribution $\mathcal{V}={\rm ker}\pi_{\ast}$, where $\pi:SO(M)\to M$, is a natural projection, is poinwise, isomorphic to the Lie algebra $\alg{so}(n)$ of the structure group $SO(n)$. Denote by $A^{\ast}$ the fundamental vertical vector field induced by an element $A\in\alg{so}(n)$. The Riemannian metric on $SO(M)$ is given as follows:
\begin{align*}
g_{SO(M)}(X^h,Y^h) &= g(X,Y),\\
g_{SO(M)}(X^h,A^{\ast}) &=0,\\
g_{SO(M)}(A^{\ast},B^{\ast}) &=-{\rm tr}(AB),
\end{align*}
where $X\in TM$, $A\in\alg{so}(n)$. Define a structure on $M$ by restricting the structure group $SO(n)$ to a subgroup $G$ such that on the level of Lie algebras, the following decomposition
\begin{equation}\label{eq:sonsplit}
\alg{so}(n)=\algg\oplus \algg^{\bot}
\end{equation}
is ${\rm ad}(G)$--invariant ($\algg^{\bot}$ denotes the orthogonal complement with respect to the Killing form). 

We say that a $G$--structure $M$ is {\it minimal} if the induced subbundle $P$ with the structure group $G$ is minimal as a submanifold inside $SO(M)$. Let us now define the condition of harmonicity of a a $G$--structure and give conditions of minimality and harmonicity via, so called, intrinsic torsion. Let $\omega$ be the connection form of the horizontal distribution $\mathcal{H}$ (induced by $\nabla$). By the invariance of the splitting \eqref{eq:sonsplit} the decomposition
\begin{equation*}
\omega=\omega_\algg+\omega_{\algg^{\bot}}
\end{equation*}
defines a connection $\omega_{\algg}$ on $P$. Denote the horizontal distribution induced by $\omega_{\algg}$ by $\mathcal{H}'$ and associated horizontal lift of $X\in TM$ by $X^{h'}$. Put
\begin{equation*}
\xi_X=-\omega_{\algg^{\bot}}(X^{h'}_p),\quad X\in T_xM,\quad \pi(p)=x.
\end{equation*} 
By ${\rm ad}(G)$--invariance of $\omega_{\algg^{\bot}}$ and the horizontal lift, it follows that $\xi_X$ is defined up to the adjoint action, thus is an element of the adjoint bundle $\algg^{\bot}_P=P\times_{{\rm ad}(G)}\algg^{\bot}$. Thus we may treat $\xi_X$ as an endomorphism $\xi_X:TM\to TM$. One can show that 
\begin{equation*}
\xi_X=\nabla^G_X-\nabla_X,
\end{equation*}  
where $\nabla^G$ is a metric connection on $M$ induced by $\omega_{\algg}$ and
\begin{equation*}
X^h=X^{h'}+(\xi_X)^{\ast}.
\end{equation*}

The reduction $P\subset SO(M)$ defines the unique section $\sigma_P$ of the associated bundle $N=SO(M)\times_{G}(SO(n)/G)$, 
\begin{equation*}
P\ni p\mapsto [[p,eG]]\in N,
\end{equation*}
where $e\in SO(N)$ id the identity element. We define a Riemannian metric on $N$ in a natural way, namely, inducing from the Riemannian metric $g$ on $M$ and restricted Killing form to $\algg^{\bot}$. We say that a $G$--structure $M$ is {\it harmonic} if the induced section $\sigma_P:M\to N$ is a harmonic section. Denote by ${\bf v}W$ the vertical component in $TN$ of a vector $W\in TN$. Then \cite{gmc1} ${\bf v}\sigma_{P\ast}(X)=-\xi_X$, thus harmonicity is coded in the intrinsic torsion. Moreover, we say than a $G$--structure is a {\it harmonic map}, if the unique section $\sigma_P$ is a harmonic map.

Let us state results obtained in \cite{gmc1} and \cite{kn} concerning harmonicity and minimality of $G$--structures. For any endomorphism $T:TM\to TM$ let
\begin{equation}\label{eq:RT}
R_T(X)=\sum_j R(e_j,T(e_j))X,\quad X\in TM.
\end{equation}  
\begin{prop}[\cite{gmc1}]\label{prop:harmonicGstructure}
A $G$--structure $M$ is harmonic if and only if the following condition holds
\begin{equation}\label{eq:harmonicGstructure}
\sum_j (\nabla_{e_j}\xi)_{e_j}=0,
\end{equation}
where $(e_j)$ is a $g$--orthonormal basis. Moroever, a $G$--structure $M$ is a harmonic map if it is a harmonic $G$--structure and 
\begin{equation*}
\sum_j R_{\xi_{e_j}}(e_j)=0.
\end{equation*}
\end{prop}
Consider a Riemannian metric $\tilde{g}$ on $M$ defined by
\begin{equation}\label{eq:tildeg}
\tilde{g}(X,Y)=g(X,Y)+\sum_j g(\xi_Xe_j,\xi_Ye_j),\quad X,Y\in TM,
\end{equation}   
where $(e_j)$ is any $g$--orthonormal basis.
\begin{prop}[\cite{kn}]\label{prop:minimalGstructure}
A $G$--structure $M$ is minimal if and only if the following condition holds
\begin{equation}\label{eq:minimalGstructure}
\sum_j (\nabla_{\tilde{e}_j}\xi)_{\tilde{e}_j}+\xi_{R_{\xi_{\tilde{e_j}}}(\tilde{e}_j)}=0,
\end{equation}
where $(\tilde{e}_j)$ is any $\tilde{g}$--orthonormal basis. Alternatively, if and only if the section $\sigma_P:M\to N$ is a harmonic map, where we consider the Riemannian metric $\tilde{g}$ instead of $g$ on $M$.
\end{prop}

\begin{rem}\label{rem:harm}
\begin{enumerate}
\item Recall that condition for harmonicity of a map $\sigma_P:(M,\tilde{g})\to N$ is of the following form
\begin{equation*}
\sum_j (\nabla_{\tilde{e}_j}\xi)_{\tilde{e}_j}=\sum_j\xi_{S_{\tilde{e}_j}\tilde{e}_j}\quad\textrm{and}\quad \sum_j R_{\xi_{\tilde{e}_j}}(\tilde{e}_j)=-\sum_j S_{\tilde{e}_j}\tilde{e}_j,
\end{equation*}
where $S$ is the difference of the Levi-Civita connection $\tilde{\nabla}$ of the metric $\tilde{g}$ and the Levi--Civita connection $\nabla$ of the metric $g$ \cite{kn}.
\item Notice that in \cite{kn} the author considered intrinsic torsion differing by the sign form the intrinsic torsion coisidered in this article and by other authors.
\end{enumerate}
\end{rem}

\section{Examples -- Locally conformally K\"ahler and contact metric structures}

In this sections we will will compute the condition \eqref{eq:harmonicGstructure} for special choices of $G$ and certain $G$--modules of the intrinsic torsion. Namely we consider $G=U(n)$, with the intrinsic torsion belonging to the class $\mathcal{W}_4$ (see \cite{gh}) and $G=U(n)\times 1$ with the intrinsic torsion belonging to the classes $\mathcal{C}_4$ and $\mathcal{C}_5$ (see \cite{cg}). Let us list common features of these considerations:
\begin{itemize}
\item The considered structures are induced by a $1$--form $\theta$ called the Lee form. Moreover, for the associated vector field $\theta^{\sharp}$ we have $\xi_{\theta^{\sharp}}=0$ (if $G=U(n)$ then, additionally, $\xi_{J\theta^{\ast}}=0$, where $J$ is an almost complex structure).
\item $\sum_j g(\xi_Xe_j,\xi_Ye_j)=\beta g(X,Y)$ for some function $\beta$ depending on $|\theta^{\sharp}|^2$ and for $X,Y$ orthogonal to $\theta^{\sharp}$ (and to $J\theta^{\sharp}$ for $G=U(n)$). Thus $\tilde{g}$ is a warped product with respect to certain distribution on $M$.
\end{itemize}

\subsection{$\mathcal{W}_4$ structures}

Let $(M,g,J)$ be a hermitian manifold, i.e., $J$ is a complex structure, $J^2=-{\rm id}_{TM}$, $J$ is $g$--invariant,
\begin{equation*}
g(JX,JY)=g(X,Y),\quad X,Y\in TM.
\end{equation*} 
Assume moreover that $M$ is locally conformally K\"ahler (LcK) \cite{va,mo}. Then, there exists one--form $\theta$, called the Lee form, such that
\begin{equation*}
d\Omega=\theta\wedge\Omega,
\end{equation*}
where $\Omega$ is a K\"ahler form, $\Omega(X,Y)=g(X,JY)$, $X,Y\in TM$. Moreover,
\begin{equation}\label{eq:nablaJ}
(\nabla_X J)Y=\frac{1}{2}\left( \theta(JY)X-\theta(Y)JX-g(X,JY)\theta^{\sharp}+g(X,Y)J\theta^{\sharp} \right).
\end{equation} 
Structure $(M,g,J)$ induces the subbundle $U(M)$ of oriented orthonormal frame bundle $SO(M)$ with the structure group $G=U(n)$, $n=\dim M$. On the level of Lie algebras we have the following splitting 
\begin{equation*}
\alg{so}(2n)=\alg{u}(n)\oplus\alg{u}(n)^{\bot},
\end{equation*}
where $\alg{u}(n)^{\bot}$ is an orthogonal complement of $\alg{u}(n)$ with respect to the Killing form on $\alg{so}(2n)$. With the identification $\alg{so}(2n)=\Lambda^2 (\mathbb{R}^{2n})^{\ast}$ we have the following descriptions
\begin{align*}
\algg &=\alg{u}(n)=\{\sigma\in \Lambda^2(\mathbb{R})^{\ast}\mid \sigma(Jv,Jw)=\sigma(v,w)\},\\
\alg{m} &=\alg{u}(n)^{\bot}=\{\sigma\in \Lambda^2(\mathbb{R})^{\ast}\mid \sigma(Jv,Jw)=-\sigma(v,w)\}.
\end{align*}
Notice that the projection ${\rm pr}_{\alg{m}}:\alg{so}(2n)\to\alg{m}$ respecting above decomposition is given by
\begin{equation*}
{\rm pr}_{\alg{m}}(A)=\frac{1}{2}\left(A+JAJ\right).
\end{equation*} 
Thus the intrinsic torsion $\xi_X$ is given by the formula \cite{ffs}
\begin{equation*}
\xi_X=-\frac{1}{2}J(\nabla_XJ).
\end{equation*}
which, by \eqref{eq:nablaJ} implies
\begin{equation*}
\xi_XY=-\frac{1}{4}\left(\theta(Y)X +\theta(JY)JX-g(X,Y)\theta^{\sharp}-g(X,JY)J\theta^{\sharp}\right).
\end{equation*}

We will compute the condition of minimality of a $G$--structure induced by LcK manifold. First of all, let us derive the formula for the Riemannian metric $\tilde{g}$. Denote by $(e_j)$ any $g$--orthonormal basis. We have
\begin{align*}
\tilde{g}(X,Y) &=g(X,Y)+\sum_j g(\xi_Xe_j,\xi_Ye_j)\\
&=g(X,Y)+\frac{1}{4}\left( g(X,Y)|\theta^{\sharp}|^2-\theta(X)\theta(Y)-\theta(JX)\theta(JY)\right)\\
&=\left(1+\frac{1}{4}|\theta^{\sharp}|^2\right)g(X,Y)
-\frac{1}{4}\left( \theta(X)\theta(Y)+\theta(JX)\theta(JY) \right).
\end{align*} 
Notice that if $X$ is orthogonal to the $J$--invariant distribution $\mathcal{D}$ spanned by the vector fields $\theta^{\sharp}$, $J\theta^{\sharp}$, then $\tilde{g}(X,Y)=\left(1+\frac{1}{4}|\theta^{\sharp}|^2\right)g(X,Y)$, whereas, if $X\in\mathcal{D}$, then $\tilde{g}(X,Y)=g(X,Y)$. Therefore $\tilde{g}$ is a warped metric with respect to the decomposition $TM=\mathcal{D}^{\bot}\oplus\mathcal{D}$. 
Thus a $\tilde{g}$--orthonormal basis $(\tilde{e}_j)$ related to $(e_j)$, where we assume $e_{2n-1}=\frac{1}{|\theta^{\sharp}|}\theta^{\sharp}$ and $e_{2n}=\frac{1}{|\theta^{\sharp}|}J\theta^{\sharp}$ is of the form
\begin{align*}
\tilde{e}_1=\frac{1}{\sqrt{1+\frac{1}{4}|\theta^{\sharp}|^2}}e_1,\ldots,\tilde{e}_{2n-2}=\frac{1}{\sqrt{1+\frac{1}{4}|\theta^{\sharp}|^2}}e_{2n-2},\tilde{e}_{2n-1}=e_{2n-1},\tilde{e}_{2n}=e_{2n}.
\end{align*}
Further, for any $X\in TM$,
\begin{equation*}
R_{\xi_X}(X)=\sum_j R(e_j,\xi_Xe_j)X=-\frac{1}{2}\left( R(\theta^{\sharp},X)X-R(J\theta^{\sharp},JX)X) \right).
\end{equation*}
Thus
\begin{equation}\label{eq:Rxi}
\sum_j R_{\xi_{\tilde{e}_j}}(\tilde{e}_j)=-\frac{1}{2}\frac{1}{1+\frac{1}{4}|\theta^{\sharp}|^2}
\left( {\rm Ric}(\theta^{\sharp})-{\rm Ric}^{\ast}(J\theta^{\sharp}) \right),
\end{equation}
where ${\rm Ric}$ is the Ricci operator and ${\rm Ric}^{\ast}$ is the $\ast$--Ricci operator defined by
\begin{equation*}
{\rm Ric}^{\ast}(X)=\sum_j R(X,Je_j)e_j,\quad X\in TM.
\end{equation*}
Before computing the remaining part of minimality condition, let us introduce one useful notion. For a vector $X\in TM$ put
\begin{equation*}
X'=\sum_j g(X,\tilde{e}_j)\tilde{e}_j.
\end{equation*}
Let us collect properties of the assignment $X\mapsto X'$ in the Proposition below.
\begin{prop}\label{prop:X'}
The following conditions hold:
\begin{enumerate}
\item $X'=\frac{1}{1+\frac{1}{4}|\theta^{\sharp}|^2}X$ if $X\in\mathcal{D}^{\bot}$,
\item $X'=X$ if $X\in\mathcal{D}$,
\item in general, $X'=\frac{1}{1+\frac{1}{4}|\theta^{\sharp}|^2}\left( X+\frac{1}{4}\left( g(X,\theta^{\sharp})\theta^{\sharp}+ g(X,J\theta^{\sharp})J\theta^{\sharp} \right) \right)$, $X\in TM$,
\item $(JX)'=JX'$ and $\theta(X')=\theta(X)$ for any $X\in TM$.
\item $g(X',Y)=g(X,Y')$ for any $X,Y\in TM$.
\end{enumerate}
\end{prop}

After lengthy computations we get
\begin{align*}
\sum_j g((\nabla_{\tilde{e}_j}\xi)_{\tilde{e}_j}Y,Z)=-\frac{1}{4}
(& (\nabla_{Z'}\theta)Y-(\nabla_{Y'}\theta)Z-(\nabla_{JZ'}\theta)JY+(\nabla_{JY'}\theta)JZ \\
&-\theta((\nabla_{JZ'}J)Y)+\theta((\nabla_{JY'}J)Z)\\
&+\theta(JY)g({\rm div}'J,Z)-\theta(JZ)g({\rm div}'J,Y) ),
\end{align*} 
where the divergence ${\rm div}'J$ equals ${\rm div}'J=\sum_j (\nabla_{\tilde{e}_j}J)\tilde{e}_j$. By \eqref{eq:nablaJ} and Proposition \ref{prop:X'} the expression
\begin{equation*}
2\theta(\nabla_{JX'}J)Y)=\theta(JX)\theta(JY)-\theta(X)\theta(Y)-g(X',Y)|\theta^{\sharp}|^2
\end{equation*}
is symmetric with respect to $X$ and $Y$. Moreover,
\begin{align*}
2{\rm div}'J &=\sum_j \left( \theta(J\tilde{e}_j)\tilde{e}_j-\theta(\tilde{e}_j)J\tilde{e}_j+|\tilde{e}_j|^2J\theta^{\sharp} \right) \\
&=-J\theta^{\sharp}-J\theta^{\sharp}+
\left(\sum_j|\tilde{e}_j|^2\right)J\theta^{\sharp}\\
&=\frac{2n-2}{1+\frac{1}{4}|\theta^{\sharp}|^2}J\theta^{\sharp}.
\end{align*}
Thus the bilinear map $(X,Y)\mapsto \theta(JX)g({\rm div}'J,Y)$ is also symmetric. Additionally, by \eqref{eq:Rxi},
\begin{align*}
\sum_j g(\xi_{R_{\xi_{\tilde{e}_j}}(\tilde{e}_j)}Y,Z)=\frac{1}{8}\frac{1}{1+\frac{1}{4}|\theta^{\sharp}|^2}( &
\theta(Y)g(\mathcal{R},Z)-\theta(Z)g(\mathcal{R},Y) \\
&-\theta(JY)\Omega(\mathcal{R},Z)+\theta(JZ)\Omega(\mathcal{R},Y)),
\end{align*}
where, to simplify notation, we put
\begin{equation*}
\mathcal{R}={\rm Ric}(\theta^{\sharp})-{\rm Ric}^{\ast}(J\theta^{\sharp}).
\end{equation*}
Concluding, by Proposition \ref{prop:minimalGstructure}, locally conformally K\"ahler structure $(M,g,J)$ with the Lee form $\theta$ is a minimal $U(n)$--structure if and only if the following condition holds
\begin{equation}\label{eq:LcKminimal}
\begin{split}
0= &(\nabla_{Z'}\theta)Y-(\nabla_{Y'}\theta)Z-(\nabla_{JZ'}\theta)JY+(\nabla_{JY'}\theta)JZ \\
&-\frac{1}{2}\frac{1}{1+\frac{1}{4}|\theta^{\sharp}|^2}(
\theta(Y)g(\mathcal{R},Z)-\theta(Z)g(\mathcal{R},Y) \\
&-\theta(JY)\Omega(\mathcal{R},Z)+\theta(JZ)\Omega(\mathcal{R},Y))
\end{split}
\end{equation}
for all $Y,Z\in TM$.  

Let us simplify condition \eqref{eq:LcKminimal}. To check validity of condition \eqref{eq:LcKminimal}, we may restrict to certain vectors $Y,Z$. Indeed, since the right hand side is skew--symmetric with respect to $Y$ and $Z$, by linearity, we have the following four possibilities:
\begin{align*}
&(i)\quad Y,Z\in\mathcal{D}^{\bot}, 
&&(ii)\quad  Y\in\mathcal{D}^{\bot}, Z=\theta^{\sharp},\\
&(iii)\quad  Y\in\mathcal{D}^{\bot}, Z=J\theta^{\sharp},
&&(iv)\quad  Y=\theta^{\sharp}, Z=J\theta^{\sharp}.
\end{align*}
In the case (iv) \eqref{eq:LcKminimal} is trivially satisfied. In the case (i) the considered condition simplifies to
\begin{equation}\label{eq:LcKminimal(4)}
0=\theta([Y,Z])-\theta([JY,JZ]),\quad Y,Z\in\mathcal{D}^{\bot}.
\end{equation}
Finally, the cases (ii) and (iii) lead to the same condition
\begin{multline}\label{eq:LcKminimal(2)}
(1+\frac{1}{4}|\theta^{\sharp}|^2)\theta(\nabla_{J\theta^{\sharp}}JY-\nabla_{\theta^{\sharp}}Y)\\
=\frac{1}{2}Y|\theta^{\sharp}|^2+\theta(\nabla_{JY}J\theta^{\sharp})-\frac{1}{2}|\theta^{\sharp}|^2g(\mathcal{R},Y),\quad Y\in \mathcal{D}^{\bot},
\end{multline}
where we used the fact that $\theta(\nabla_Y\theta^{\sharp})=\frac{1}{2}Y|\theta^{\sharp}|^2$. Thus we have proved the following result. 
\begin{thm}\label{thm:LcKminimal}
A LcK manifold $(M,g,J)$ is minimal as a $U(n)$--structure if and only if \eqref{eq:LcKminimal(4)} and \eqref{eq:LcKminimal(2)} hold.
\end{thm}

\begin{exa}
Let $(M,g)$ be the Euclidean space $\mathbb{R}^{2n}$ with the canonical complex structure $J$. Let $f$ be arbitrary smooth function on $M$ and consider the conformal deformation $g_0=e^{-2f}g$. We will compute the condition of minimality of $(M,g_0,J)$. In this case we have globally conformally K\"ahler manifold. Recall that the curvature tensor $R_0$ is given by the formula
\begin{align*}
g_0(R_0(X,Y)Z,W)= & L(X,Z)g_0(Y,W)+L(Y,W)g_0(X,Z)\\
&-L(X,W)g_0(Y,Z)-L(Y,Z)g_0(X,W)\\
&-e^{4f}|df|_0^2(g_0(X,Z)g_0(Y,W )-g_0(Y,Z)g_0(X,W)),
\end{align*}
where $L(X,Y)=(\nabla_Xdf)Y+df(X)df(Y)$ and hessian is computed with respect to the Levi--Civita connection of the Euclidean metric $g$ \cite{vl}. Moreover, $\theta=df$ is the Lee form. Then, simple calculations leed to the equality
\begin{equation*}
g_0(\mathcal{R},Y)=e^{-2f}\left((3-2n)L(\theta^{\sharp},Y)-L(J\theta^{\sharp},JY)\right),\quad Y\in \mathcal{D}^{\bot}. 
\end{equation*}
Since $\theta=df$, then $d\theta=0$ and condition \eqref{eq:LcKminimal(4)} holds trivially. Moreover $d\theta=0$ implies the following relations
\begin{equation*}
\theta(\nabla_{\theta^{\sharp}}Y)=-\frac{1}{2}Y|\theta^{\sharp}|_0^2\quad\textrm{and}\quad
\theta(\nabla_{J\theta^{\sharp}}Y)=\theta(\nabla_YJ\theta^{\sharp})\quad\textrm{for $Y\in\mathcal{D}^{\bot}$}.
\end{equation*}
Consequently, \eqref{eq:LcKminimal(2)} reduces to
\begin{multline}\label{eq:LcKexample}
\theta(\nabla_{JY}J\theta^{\sharp})+\frac{1}{2}Y|\theta^{\sharp}|^2_0\\
=(1+\frac{1}{4}|\theta^{\sharp}|^2_0)e^{-2f}\left( (3-2n)L(\theta^{\sharp},Y)+L(J\theta^{\sharp},JY) \right)
,\quad Y\in\mathcal{D}^{\bot}.
\end{multline} 
By the fact that $L(\theta^{\sharp},Y)=\frac{1}{2}Y|\theta^{\sharp}|^2$ and $L(J\theta^{\sharp},JY)=-\theta(\nabla_{JY}J\theta^{\sharp})$ for $Y\in\mathcal{D}^{\bot}$, condition \eqref{eq:LcKexample} takes the form
\begin{equation*}
(1+e^{-2f}(1+\frac{1}{4}|\theta^{\sharp}|^2_0)\theta(\nabla_{JY}J\theta^{\sharp})=((3-2n)e^{-2f}(1+\frac{1}{4}|\theta^{\sharp}|^2_0-1)Y|\theta^{\sharp}|^2,\quad Y\in\mathcal{D}^{\bot}.
\end{equation*} 
The above condition is has a solution for some $f$. For example, we may choose $f=(\varphi(x_1),0,\ldots,0)$, where $\varphi$ is some smooth non--vanishing function. Hence there is globally conformally K\"ahler manifold, which is minimal as a $U(n)$--structure and for which the Lee form is not parallel. 
\end{exa}

\begin{thm}\label{thm:Vaisman}
Assume $(M,g,J)$ is locally conformally K\"ahler manifold with parallel Lee form $\theta$. If the $U(n)$--structure induced by $M$ is a harmonic map, then it is a minimal $U(n)$--structure. In particular, Hopf manifolds induce minimal $U(n)$--structures.
\end{thm}
\begin{proof}
Assume $\theta^{\sharp}$ is parallel. Notice that by \eqref{eq:nablaJ}
\begin{equation*}
\nabla_XJ\theta^{\sharp}=(\nabla_X J)\theta^{\sharp}=\left\{\begin{array}{rcl}
-\frac{1}{2}|\theta^{\sharp}|^2JX & \textrm{for} & X\in\mathcal{D}^{\bot}\\
0 & \textrm{for} & X\in\mathcal{D}
\end{array}
\right.
\end{equation*}  
and $\xi_{\theta^{\sharp}}=\xi_{J\theta^{\sharp}}=0$. Put, for simplicity, $c=\frac{1}{1+\frac{1}{4}|\theta^{\sharp}|^2}$. Then
\begin{equation*}
\sum_j (\nabla_{\tilde{e}_j}\xi)_{\tilde{e}_j}=c\sum_j (\nabla_{e_j}\xi)_{e_j}+\frac{1-c}{|\theta^{\sharp}|^2}\left( (\nabla_{\theta^{\sharp}}\xi)_{\theta^{\sharp}}+(\nabla_{J\theta^{\sharp}}\xi)_{J\theta^{\sharp}} \right)
=c\sum_j (\nabla_{e_j}\xi)_{e_j}.
\end{equation*}
Moreover, by \eqref{eq:Rxi}, we have
\begin{equation*}
\sum_j R_{\xi_{\tilde{e}_j}}(\tilde{e}_j)=c\sum_j R_{\xi_{e_j}}(e_j).
\end{equation*}
Assuming $M$ that induces a $U(n)$--structure, which is a harmonic map, then (see Proposition \ref{prop:harmonicGstructure}) $\sum_j (\nabla_{e_j}\xi)_{e_j}=0$ and $\sum_j R_{\xi_{e_j}}(e_j)$. Thus, by above considerations, $\sum_j (\nabla_{\tilde{e}_j}\xi)_{\tilde{e}_j}=0$ and $\sum_j R_{\xi_{\tilde{e}_j}}(\tilde{e}_j)=0$. In particular, by Proposition \ref{prop:minimalGstructure}, $M$ induces minimal $U(n)$--structure. Thus we have proved the first part of the theorem. The second part follows by the fact that Hopf manifolds are examples of LcK manifolds, which are harmonic structures \cite{gmc1}.
\end{proof}

\subsection{$\mathcal{C}_4$ and $\mathcal{C}_4$ structures}

Let $(M,g,\varphi,\eta,\zeta)$ be an almost contact metric structure (of dimension $2n+1$), i.e., the Riemannian metric $g$, endomorphism $\varphi:TM\to TM$, one--form $\eta$ and a vector field $\zeta$ satisfy the following conditions
\begin{align*}
&\varphi^2=-{\rm Id}_{TM}+\eta\otimes\zeta, && \eta=\zeta^{\flat},\\
& g(\varphi X,\varphi Y)=g(X,Y)-\eta(X)\eta(Y) && |\zeta|^2=1.
\end{align*}
Then $\varphi$ defines almost complex structure on the distribution $\mathcal{E}$ orthogonal to unit vector field $\zeta$. We call $\zeta$ the Reeb field. Such conditions imply reduction of the structure group of the oriented orthonormal frame bundle to $G=U(n)\times 1\subset SO(2n+1)$. On the level of Lie algebras we have
\begin{equation*}
\alg{so}(2n+1)=\algg\oplus\alg{m},
\end{equation*} 
where $\algg$ is isomorphic to $\alg{u}(n)$ and its orthogonal complement $\alg{u}(n)^{\bot}$ equals $\alg{m}$. By the identification $\alg{so}(2n+1)=\Lambda^2(\mathbb{R}^{2n+1})^{\ast}$ we have
\begin{align*}
\algg &=\{\sigma\mid \sigma(\varphi X,\varphi Y)=\sigma({\rm pr}X,{\rm pr}Y)\},\\
\alg{m}&=\{\sigma\mid \sigma(\varphi X,\varphi Y)=-\sigma({\rm pr}X,{\rm pr}Y)\}\oplus \mathbb{R}\eta\wedge\eta^{\bot},
\end{align*}
where ${\rm pr}X=X-\eta(X)\zeta$ is the orthogonal projection onto $\mathcal{E}$ \cite{gmc2}.
The projection ${\rm pr}_{\alg{m}}:\alg{so}(2n+1)\mapsto\alg{m}$ respecting the above decomposition is given by
\begin{equation*}
{\rm pr}_{\alg{m}}(A)=\frac{1}{2}\left( A+\varphi A\varphi+\eta A\otimes\zeta+\eta\otimes A\zeta \right).
\end{equation*}
Thus the intrinsic torsion equals \cite{gmc2}
\begin{equation}\label{eq:intrinsictorsioncontact}
\xi_XY=\frac{1}{2}(\nabla_X\varphi)\varphi Y+\frac{1}{2}(\nabla_X\eta)Y\,\zeta-\eta(Y)\nabla_X\zeta.
\end{equation}
Recall the following identities \cite{cg} 
\begin{equation}\label{eq:contactidentities}
\begin{split}
(\nabla_X\eta)Y &=g(Y,\nabla_X\zeta)=(\nabla_X\Phi)(\zeta,\varphi Y),\\
(\nabla_X\Phi)(Y,Z) &=g(Y,(\nabla_X\varphi)Z),
\end{split}
\end{equation}
where $\Phi(X,Y)=g(X,\varphi Y)$ is the fundamental $2$--form. Assume that almost contact metric structure $M$ is locally conformally integrable with the Lee form $\theta$. If $\theta=\alpha\eta$ for some smooth function $\alpha$, then the intrinsic torsion belongs to the class $\mathcal{C}_5$ of the space of possible intrinsic torsions, whereas if $\eta(\theta^{\sharp})=0$, then $\xi_X$ is in the class $\mathcal{C}_4$ \cite{cg}. 

We will give example of minimal $U(n)\times 1$--structure for the intrinsic torsion belonging to the class $\mathcal{C}_4$ by comparison with the Grey--Hervella class $\mathcal{W}_4$ and then we will concentrate in details on the class $\mathcal{C}_5$. 

First, notice that if $(M,g,J)$ is an almost hermitian manifold, then $\bar{M}=M\times\mathbb{R}$ becomes almost contact metric structure by putting (see \cite{cg})
\begin{align*}
&\varphi\left(X,a\frac{d}{dt}\right)=(JX,0), && \bar{g}\left(\left(X,a\frac{d}{dt}\right),\left(Y,b\frac{d}{dt}\right)\right)=g(X,Y)+ab,\\
& \zeta=\frac{d}{dt}, && \eta\left(X,a\frac{d}{dt}\right)=a.
\end{align*}
Then $\xi$ and $\bar{\xi}$ are related as follows
\begin{equation*}
\bar{\xi}_{(X,0)}(Y,0)=(\xi_XY,0)
\end{equation*}
and $\bar{\xi}$ vanishes for other possible choices. Thus, $\xi\in\mathcal{W}_4$ if and only if $\bar{\xi}\in\mathcal{C}_4$. Moreover, the Riemannian metric $\tilde{\bar{g}}$ equals $\tilde{g}$ on $M$. By the flatness of $\mathbb{R}$ we get that the $U(n)$--structure induced by such $M$ is minimal if and only if $U(n)\times 1$--structure induced by $\bar{M}$ is minimal. Hence, the product manifold $M\times{R}$, where $M$ is a Hopf manifold is a minimal $U(n)\times 1$--structure. 

Assume now $\xi\in\mathcal{C}_5$. Then $M$ is called $\alpha$--{\it Kenmotsu} and (see \cite{cg})
\begin{equation}\label{eq:alphaKenmotsueq}
(\nabla_X\Phi)(Y,Z)=-\alpha(\Phi(X,Z)\eta(Y)-\Phi(X,Y)\eta(Z)).
\end{equation}
Comparing \eqref{eq:alphaKenmotsueq} and \eqref{eq:intrinsictorsioncontact} we have
\begin{align*}
(\nabla_X\eta)Y &=\alpha(g(X,Y)-\eta(X)\eta(Y)),\\
(\nabla_X\varphi)Y &=-\alpha(\Phi(X,Y)\zeta+\eta(Y)\varphi X),\\
\nabla_X\zeta &=\alpha(X-\eta(X)\zeta)=\alpha{\rm pr}X.
\end{align*}
Hence, using \eqref{eq:contactidentities} we get the formula for the intrinsic torsion
\begin{equation*}
\xi_X Y=\alpha(g(X,Y)\zeta-\eta(Y)X),\quad X,Y\in TM.
\end{equation*}
Then
\begin{equation*}
\xi_{\zeta}=0,\quad \xi_X\zeta=-\alpha{\rm pr}X.
\end{equation*}
By a simple computation we also get
\begin{equation*}
\tilde{g}(X,Y)=g(X,Y)+\sum_j g(\xi_Xe_j,\xi_Ye_j)
=(1+2\alpha^2)g(X,Y)-2\alpha^2\eta(X)\eta(Y),
\end{equation*}
where $(e_j)$ is a $g$--orthonormal basis. Therefore, the associated $\tilde{g}$--orthonormal basis $(\tilde{e}_j)$, where we assume that $e_{2n+1}=\zeta$, is given by
\begin{equation*}
\tilde{e}_1=\frac{1}{\sqrt{1+2\alpha^2}}e_1,\ldots,
\tilde{e}_{2n}=\frac{1}{\sqrt{1+2\alpha^2}}e_{2n},
\tilde{e}_{2n+1}=\zeta.
\end{equation*} 
Analogously as in the hermitian case, for a vector $X\in TM$ put
\begin{equation*}
X'=\sum_j g(X,\tilde{e}_j)\tilde{e}_j.
\end{equation*}
Then $X'=\frac{1}{1+2\alpha^2}X+\frac{2\alpha^2}{1+2\alpha^2}\eta(X)\zeta$, which implies
\begin{equation*}
X'=\frac{1}{1+2\alpha^2}X\quad\textrm{for $X\in\mathcal{E}$}\quad\textrm{and}\quad \zeta'=\zeta.
\end{equation*}
Now we may turn to computing the condition of minimality of $U(n)\times 1$--structure. We have $R_{\xi_X}(X)=2\alpha R(\zeta,X)X$, thus
\begin{equation*}
\sum_j R_{\xi_{\tilde{e}_j}}(\tilde{e}_j)= \frac{2\alpha}{1+2\alpha^2}{\rm Ric}(\zeta).
\end{equation*}
Hence,
\begin{equation*}
\sum_j g(\xi_{R_{\xi_{\tilde{e}_j}}(\tilde{e}_j)}Y,Z)=\frac{2\alpha^2}{1+2\alpha^2}(\eta(Y){\rm Ric}(\zeta,Z)-\eta(Z){\rm Ric}(\zeta,Y)).
\end{equation*}
Moreover,
\begin{align*}
(\nabla_X\xi)_XY &=(X\alpha)(g(X,Y)\zeta-\eta(Y)X)+\alpha(g(X,Y)\nabla_X\zeta-(\nabla_X\eta)Y\cdot X)\\
&=(X\alpha)(g(X,Y)\zeta-\eta(Y)X)+\alpha^2\eta(X)(\eta(Y)X-g(X,Y)\zeta),
\end{align*}
which implies
\begin{equation*}
\sum_j g((\nabla_{\tilde{e}_j}\xi)_{\tilde{e}_j}Y,Z)=\eta(Z)Y'\alpha-\eta(Y)Z'\alpha.
\end{equation*}
Thus we have obtained the following observation. $\alpha$--Kenmotsu manifold is minimal as an $U(n)\times 1$--structure if and only if for any $Y,Z\in TM$ the following condition holds
\begin{equation}\label{eq:Kenmotsuminimal}
0=\eta(Z)Y'\alpha-\eta(Y)Z'\alpha
+\frac{2\alpha^2}{1+2\alpha^2}(\eta(Y){\rm Ric}(\zeta,Z)-\eta(Z){\rm Ric}(\zeta,Y)).
\end{equation}
Let us simplify condition \eqref{eq:Kenmotsuminimal}. For $Y,Z\in\mathcal{E}$ \eqref{eq:Kenmotsuminimal} holds trivially,
whereas for $Y\in\mathcal{E}$ and $Z=\zeta$ we obtain
\begin{equation*}
0=Y\alpha-2\alpha^2{\rm Ric}(\zeta,Y).
\end{equation*}
Concluding we may state the following corollary.
\begin{thm}\label{thm:Kenmotsuminimal}
$\alpha$--Kenmotsu manifold is minimal as an $U(n)\times 1$--structure if and only if
\begin{equation*}
Y\alpha=2\alpha^2{\rm Ric}(\zeta,Y),\quad Y\in\mathcal{E}.
\end{equation*}
\end{thm}

\begin{cor}
Kenmotsu manifolds, which satisfy ${\rm Ric}(\zeta,Y)=0$ for $Y\in\mathcal{E}$ are minimal $U(n)\times 1$--structures.
\end{cor} 
\begin{proof}
By definition Kenmotsu manifold is an $\alpha$--Kenmotsu manifold with $\alpha=1$. Hence $Y\alpha=0$.
\end{proof}
  
Let us finish by giving one example. 
\begin{exa}
Consider the hyperbolic space $H^{2n+1}=\{(x_1,\ldots,x_{2n+1})\mid x_1>0\}$, where the Riemannian metric $g$ is of the form
\begin{equation*}
g=\frac{1}{c^2x_1^2}\sum_j dx_j^2
\end{equation*}
for some non--zero constant $c$. One can show that $H^{2n+1}$ is of constant sectional curvature $-c^2$ and induces $\alpha$--Kenmotsu structure, with $\zeta=cx_1\frac{\partial}{\partial x_1}$ and $\alpha=-c$ \cite{cg,gmc2}. Since $H^{2n+1}$ is a space form it follows that ${\rm Ric}(\zeta,Y)=0$ for $Y$ orthogonal to $\zeta$. Thus by Theorem \ref{thm:Kenmotsuminimal} $H^{2n+1}$ is a minimal $U(n)\times 1$--structure.
\end{exa}

\end{document}